\newif\ifjsc\jsctrue

\documentclass[final,1p,times,authoryear]{elsarticle}

\usepackage[utf8]{inputenc}

\usepackage{geometry}

\usepackage{enumitem}

\usepackage{setspace}

\usepackage{amsmath}
\usepackage{amssymb}
\usepackage{amsfonts}
\usepackage{amsthm}

\usepackage{mathrsfs}

\usepackage{array}

\usepackage{xcolor}

\newcommand{\pcite}[1]{\citeparen{\cite{#1}}}
\newcommand{\citeparen}[1]{#1}
\ifjsc
    \newcommand{\mybibliography}{\bibliography{refs-bibtex}}
    \renewcommand{\citeparen}[1]{(#1)}
\else
    \usepackage[isbn=false,giveninits=true]{biblatex}
    \newcommand{\mybibliography}{\printbibliography}
\fi

\usepackage[page,toc]{appendix}

\usepackage[T1]{fontenc}
\usepackage{lmodern}

\usepackage{upquote}

\usepackage{listings}

\usepackage{longtable}

\usepackage{afterpage}

\usepackage{algpseudocode}
\makeatletter
\algnewcommand{\ParState}[1]{\State\parbox[t]{\dimexpr\linewidth-\the\ALG@thistlm}{\strut #1\strut}}
\makeatother
\algnewcommand{\LineComment}[1]{\vspace{1em}\ParState{\textit{(#1)}}}
\algnewcommand{\Assert}[0]{\textbf{assert} }

\usepackage{tikz,pgfplots}
\usetikzlibrary{fit,arrows,arrows.meta,positioning,tikzmark}
\pgfplotsset{compat=1.3}

\ifjsc\else
    \usepackage{caption}
    \DeclareCaptionFormat{myformat}{\textbf{#1#2}#3\hrulefill}
\fi

\usepackage{rotating}

\usepackage{adjustbox}

\usepackage{hyperref}

\ifjsc\else
\fi

\newcommand{\code}[1]{%
\begingroup%
\let\oldunderscore\_%
\renewcommand{\_}{\discretionary{\oldunderscore}{}{\oldunderscore}}%
\renewcommand{\-}{\discretionary{\textrm{-}}{}{}}%
\texttt{#1}%
\endgroup}

\newcommand{\define}[1]{\textbf{#1}}

\newcommand{\FF}{\mathbb{F}}
\newcommand{\NN}{\mathbb{N}}

\newcommand{\QQ}{\mathbb{Q}}
\newcommand{\RR}{\mathbb{R}}
\newcommand{\ZZ}{\mathbb{Z}}

\newcommand{\cO}{\mathcal{O}}

\newcommand{\cX}{\mathcal{X}}

\newcommand{\val}{\operatorname{val}}

\def\comment{}
\def\endcomment{}
\long\def\comment#1\endcomment{}

    \theoremstyle{plain}
    \ifjsc
        \newtheorem{lemma}{Lemma}
    \else
        \newtheorem{lemma}{Lemma}[section]
    \fi

    \theoremstyle{definition}
    \newtheorem{algorithm}[lemma]{Algorithm}
    \newtheorem{definition}[lemma]{Definition}
    \newtheorem{example*}[lemma]{Example} 
    \newenvironment{example}    
      {%
       \pushQED{\qed}\begin{example*}}
      {\popQED\end{example*}}
    \theoremstyle{remark}
    \newtheorem{remark}[lemma]{Remark}
\newcommand{\qedabove}{\vspace*{-1.3\baselineskip}\qedhere}

\makeatletter
\renewcommand*{\verbatim@font}{\ttfamily\fontseries{m}\selectfont}
\makeatother

\lstdefinelanguage{Magma}{
  morekeywords={end,function,intrinsic,procedure,for,while,repeat,until,do,in,if,else,elif,then,error,assert,require,when,where,is,print,printf,vprint,vprintf,time,declare,verbose,type,attributes,return,continue,break,delete,loop},
  morekeywords=[2]{eq,ne,le,lt,ge,gt,cmpeq,cmpne,not,notin,and,or,notsubset,subset,meet,join,diff,sdiff,assigned,eval},
  morekeywords=[3]{sub,ncl,func,proc,ideal,elt},
  morekeywords=[4]{AnyPadExact,StrPadExact,PadExactElt,FldPadExact,FldPadExactElt,RngUPol_FldPadExact,RngUPolElt_FldPadExact,RngMPol_FldPadExact,RngMPolElt_FldPadExact,SetCart_PadExactElt,Tup_PadExactElt,Val_PadExactElt,Val_FldPadElt,Val_RngUPolElt_FldPad,Val_RngMPolElt_FldPad,RngInt,RngIntElt,SetCart,Tup,List,FldNum,FldNumElt,FldRat,FldRatElt,FldPad,FldPadElt,Getter,BoolElt,SetCart_PadExact,Tup_PadExact},
  sensitive=true,
  morecomment=[l]{//},
  morecomment=[s]{/*}{*/},
  morestring=[b]",
}
\begin{document}

\begin{frontmatter}

\title{Exact \(p\)-adic computation in Magma{\let\thefootnote\relax\footnotetext{\copyright{} 2020. This manuscript version is made available under the CC-BY-NC-ND 4.0 license \url{http://creativecommons.org/licenses/by-nc-nd/4.0/}}}}

\author{Christopher Doris}
\address{Heilbronn Institute for Mathematical Research, School of Mathematics, University of Bristol, BS8 1UG, UK}
\ead{christopher.doris@bristol.ac.uk}
\ead[url]{https://cjdoris.github.io}


\begin{abstract}
We describe a new arithmetic system for the Magma computer algebra system for working with \(p\)-adic numbers exactly, in the sense that numbers are represented lazily to infinite \(p\)-adic precision. This is the first highly featured such implementation. This has the benefits of increasing user-friendliness and speeding up some computations, as well as forcibly producing provable results. We give theoretical and practical justification for its design and describe some use cases. The intention is that this article will be of benefit to anyone wanting to implement similar functionality in other languages.
\end{abstract}

\end{frontmatter}

\section{Introduction}
\label{sec-intro}

When dealing with completed fields, such as \(\RR\) or \(\QQ_p\), it is generally quite difficult to represent elements exactly. Instead, the commonest way to represent elements is by specifying them to some pre-determined precision, and then performing operations such as arithmetic to this precision also. This is the foundation of floating point arithmetic. For example, one might represent the real number \(e\) by its approximation \(2.718281828\) to a precision of 10 real digits. We say such a representation is \define{inexact} because several real numbers can have the same representation: \(e\), \(2.718281828\) and \(2.7182818281\) all have the same representation to 10 digits precision.

Such a representation is also usually \define{zealous} meaning that when an operation is performed, such as multiplication, it is immediately computed to the required precision. For instance, computing \(e \times e\) will work to 10 digits precision and actually compute \(2.718281828 \times 2.718281828 = 7.389056096\). In fact, \(e \times e = 7.389056098\ldots\), demonstrating that precision errors can creep into the results, so that they are in fact less precise than the precision claims.

An often-suggested alternative to zealous arithmetic is \define{lazy arithmetic}, in which an operation does not produce an answer per-se, but a ``promise to produce an answer to a desired precision''. That is, calling \(e \times e\) would not produce the approximation \(7.389056096\), but would produce a function which, when called with an integer \(k\), returns an approximation to \(e \times e\) to \(k\) digits precision.

Such a function can be said to be an \define{exact} representation of a real number, because no two distinct real numbers have the same representation: for a sufficiently large precision \(k\), the representing functions will return different approximations.

These comments hold true for \(p\)-adic numbers too. For instance, an element of \(\QQ_p\) is generally represented in zealous, inexact arithmetic by its residue class in \(\QQ_p / p^k \ZZ_p\) for some absolute precision \(k\): e.g. \(1 + 2^{10} \ZZ_2\) might represent \(1\), \(1 + 2^{10}\) or \(1 + 5 \times 2^{100}\).

There are numerous implementations of such zealous, inexact \(p\)-adic arithmetic. FLINT \pcite{flint} provides some low-level arithmetic with elements of \(\QQ_p\), univariate polynomials over \(\QQ_p\), and unramified extensions of \(\QQ_p\). Sage \pcite{sage} and Magma \pcite{magma} have more fully-featured implementations, including arbitrary finite extensions of \(\QQ_p\) and higher-level routines for tasks such as factoring.

Also of note is an implementation in Mathemagix \pcite{mathemagix} of the so-called \define{relaxed \(p\)-adic arithmetic} \pcite{relaxed,relaxedhensel}, which treats elements of \(\QQ_p\) like an infinite sequence of \(p\)-adic coefficients, somewhat like \(\FF_p((t))\), and provides algorithms to lazily compute these coefficients one at a time. The latter property makes relaxed arithmetic particularly useful for \(p\)-adic recursion solving. In principle it is useful in general but is somewhat more complicated to implement than the lazy arithmetic presented in this article and as such is less fully featured.

A more in-depth description of different \(p\)-adic arithmetic systems is given by \cite{caruso}.

In this article, we present a new implementation of a lazy, exact \(p\)-adic arithmetic system for the Magma computer algebra system \pcite{magma}. As mentioned above, Magma already has a fully-featured implementation of zealous, inexact \(p\)-adic arithmetic. Our system uses this inexact functionality already available as much as possible, allowing for rapid addition of new features to the exact arithmetic as soon as they are available inexactly.

To the author's knowledge, this is the first highly-featured, general-purpose implementation of lazy \(p\)-adic arithmetic.

This article describes the rationale and the fundamental concepts behind the implementation, but does not constitute a user manual. The implementation will be available in an upcoming release of Magma with documentation available in the Magma Handbook.

\subsection{Prototypes}

While this implementation is written in C, we have also produced two prototype exact \(p\)-adic arithmetic systems written in Magma \citep{exactpadicscode,exactpadics,exactpadics2code}. Although much less performant, the development of these prototypes heavily informed the present implementation. In particular, the present implementation is very similar to the second prototype, and the rationale for not following the more obvious system in the first prototype is given in Section~\ref{sec-epoch-vs-prec}.

To demonstrate the utility of exact arithmetic, our prototype packages have been used to implement the algorithm \ifjsc of\else in\fi{} \cite{conductor} to compute the 2-part of the conductor of a hyperelliptic curve of genus 2 defined over a number field. This implementation is available from \cite{conductorcode}. It uses such high-level \(p\)-adic routines as: computing the completion of a number field at a finite place; computing the factorization of a univariate polynomial and the fields defined by its factors; and Hensel-lifting roots of a system of multivariate equations.

As another application, the prototype packages can optionally be used with the implementation of the algorithms of \cite{galoisgroups} for computing the Galois group of a \(p\)-adic polynomial. This is available from \cite{galoiscode}. With either package present the Galois group algorithm becomes provably correct, whereas otherwise with inexact \(p\)-adics there is no such guarantee. We also find that the algorithms run faster with exact \(p\)-adics, at least for reasonably high-degree inputs.

It is work in progress to translate these algorithms to the new system.

\subsection{Notation}

The field of \(p\)-adic numbers is denoted \(\QQ_p\) and its ring of integers is \(\ZZ_p\). If \(K/\QQ_p\) is a finite extension, then \(\cO_K\) is its ring of integers, \(\pi_K \in \QQ_p\) is a uniformizing element, and \(\val_K\) is the \(\pi_K\)-adic valuation such that \(\val_K(\pi_K)=1\). We drop the subscript when it is clear from context.

\subsection{Structure of this article}

We motivate the development of exact arithmetic in Section~\ref{sec-compare} by comparing it with zealous arithmetic. After a brief description in Section~\ref{sec-inexact} of the inexact arithmetic already available in Magma, we describe our new system in Section~\ref{sec-exact}. Next, in Section~\ref{sec-rationale} we give a rationale for our design, arguing based on simple complexity analysis as well as experience, then in Section~\ref{sec-timings} we test the performance of our system in practice and compare to inexact arithmetic. Finally, in Section~\ref{sec-algorithms}, we describe a few algorithms which highlight some of the differences in implementing functionality for exact arithmetic as compared to inexact arithmetic.

\section{Motivation: Comparison of zealous and lazy arithmetic}
\label{sec-compare}

\subsection{Precision guessing}

In zealous arithmetic, the user is generally required to choose a precision to work at in advance. Then all computations are performed to that precision, and it may happen that the precision chosen was not sufficient. In this case, the user will probably start the computation over with a higher precision. This process of manually increasing the precision of a computation can be burdensome for the user. In lazy arithmetic, such precision decisions are made automatically as far as possible.

\begin{example}
Here is a typical interactive Magma session, using its builtin inexact arithmetic:
\begin{lstlisting}
> // try to factorize at precision 10
> K := pAdicField(2, 10);
> R<x> := PolynomialRing(K);
> f := my_favourite_polynomial(R);
> Factorization(f);
error: ...
> // try to factorize at precision 20
> K := pAdicField(2, 20);
> R<x> := PolynomialRing(K);
> f := my_favourite_polynomial(R);
> Factorization(f);
error: ...
> // try to factorize at precision 40
> K := pAdicField(2, 40);
> R<x> := PolynomialRing(K);
> f := my_favourite_polynomial(R);
> Factorization(f);
[ <x^10 + ... >, ... ]
\end{lstlisting}

Using lazy arithmetic provided by our package, the equivalent session would be the following. Note that there is no explicit mention of precision.
\begin{lstlisting}
> K := pAdicField(2 : Exact);
> R<x> := PolynomialRing(K);
> f := my_favourite_polynomial(R);
> Factorization(f);
[ <x^10 + ... >, ... ]
\end{lstlisting}
\qedabove
\end{example}

\subsection{Repeated computation}

Some \(p\)-adic computations fall into two distinct parts: the first part being computationally expensive but independent of the precision, and the second part being computationally cheap. For example, polynomial root-finding algorithms typically have an expensive first part to identify the roots, and then use cheap Hensel lifting in the second part to approximate these roots to high precision.

If such a computation is performed with inexact arithmetic at several increasing precisions, as in the previous section, then both parts will be repeated at each precision.

If we instead use exact arithmetic, the expensive first part is only performed once, and only the cheap second part is repeated for each successive precision.

\subsection{Locality of precision}

In lazy arithmetic, each individual computation is performed to approximately the smallest precision it can be, and so precisions are very local in the computation. In zealous arithmetic, the precision is generally chosen once at the start of a computation, so each operation is performed to the same precision, and so precisions are more global. If there is a single operation requiring a high global precision, this increases the precision that all other operations are performed to, which is a performance hit compared to lazy arithmetic.

\begin{example}
An example comes from the conductor algorithm mentioned in the introduction. One portion of this algorithm takes a polynomial \(f(x) \in \QQ_2[x]\), computes its factorization, chooses a factor \(g(x)\), computes the extension \(K/\QQ_2\) defined by \(g\), and then finds a root of \(g\) in \(K\). Usually, the precision required for the factorization far exceeds that of the root-finding; however, because the root-finding is over an extension \(K\), if it were to be done at the same high precision as the factorization, its run-time would often dominate.
\end{example}

\subsection{Correctness and provability}

When a \(p\)-adic number \(x \in K\) is represented inexactly as a class \(x + \pi^k \cO\), then it can be ambiguous whether it is really representing \(x\) or the class itself. For many operations, the distinction makes no difference; for example since \[(x+y)+\pi^k \cO = (x+\pi^k \cO) + (y + \pi^k \cO)\] then addition works the same in either interpretation. For other operations, Magma can produce potentially misleading answers; for example if \(x\) is represented as \(0 + \pi^k \cO\) then \(\code{Valuation}(x)\) will return \(k\), when in fact all we really know is that \(\val(x) \geq k\).

Our implementation avoids this ambiguity. Whenever we provide a function whose name is a mathematical concept (such as \code{Valuation}) then the returned value will be correct, and in particular will not depend on the precision to which the input was given --- instead, more precision will be computed as required. This frees the user to think about these objects as actual \(p\)-adic numbers, not as residue classes.

Similarly, if \(x=1\) and \(y=1+2^{10}\) are both represented inexactly as \(1+2^{10}\ZZ_2\), then they will be reported as equal, even though the true values are different. In fact, it is not possible to determine that two arbitrary \(p\)-adic numbers are equal when given to any finite precision, because they could always differ at some higher precision. With our exact arithmetic, \(x\) and \(y\) would be reported as unequal by repeatedly increasing their precision until they differ. If also \(z=1\), then testing if \(x\) and \(z\) are equal will most likely not terminate, but it will certainly not return false.

In some cases, such as polynomial root-finding and factorization (\code{Roots} and \code{Fact\-or\-iz\-ation}), the correctness of the output is forced by the fact that the outputs are given exactly. That is, if \code{Roots} returns a root (exactly), then it by definition comes with a program to compute an approximation to the root to arbitrarily high precision, and therefore assuming the program is correct this is a proof that the root is correct. In the case of \code{Roots}, it is Hensel's lemma which provides this proof.

Any functions we provide whose output depends on the representation itself are given names which make this clear. The terms \code{Weakly} and \code{Definitely} are used to denote tests which can give false positives or false negatives; for example \(\code{IsWeaklyZero}(x)\) is true if its input appears to be zero up to some precision (but does not guarantee it is zero). Similarly the term \code{Weak} denotes functions whose output depends on the precision of its inputs, so for example \(\code{WeakValuation}(x)\) returns a lower bound on \(\val(x)\) dependent on the current precision of \(x\).

\subsection{Overheads}

The main down-sides of lazy arithmetic are the extra time and memory overheads introduced. In lazy arithmetic, \(p\)-adic values depend on other \(p\)-adic values, and all these dependencies need to be kept in memory for the duration of a computation. Each time an operation is performed, some dependency tracking and propagation needs to occur, which entails some processing time overhead.

This said, we find that these overheads are usually negligible unless one performs a large number of ordinarily very fast operations, such as basic arithmetic, and even then the overheads do not dominate (see Section~\ref{sec-timings}). If this is the case, then one can consider implementing the whole sequence of operations as a new atomic \(p\)-adic operation, which therefore now only contributes a single node to the dependency graph.

\section{Inexact \texorpdfstring{\(p\)}{p}-adics in Magma}
\label{sec-inexact}

Before describing our implementation, we give a brief description of the inexact \(p\)-adics already available in Magma. In this discussion, \(K/\QQ_p\) is a finite extension and \(\cO/\ZZ_p\) is its ring of integers.

The inexact representation of \(K\) has type \code{FldPad} and \(\cO\) has type \code{RngPad}. If \(K=\QQ_p\) then both are defined simply by the prime \(p\). Otherwise, \(K\) is defined by a polynomial over some sub-field \(M\). The coefficients of this polynomial are themselves inexact.

Alternatively, \(K\) may be defined by a function \(m : \NN \to M[x]\), and where \(m(k)\) is the defining polynomial of \(K/M\) given to absolute precision \(k\). We refer fields defined in this way as \define{semi-exact} because their defining polynomial is defined exactly via \(m\) but the elements are still inexact.

One can change the precision of such a semi-exact field, yielding a field defined by a polynomial to that precision. Importantly, Magma understands these fields to be compatible with each other, allowing for very cheap coercion of elements between them.

Integer elements \(x \in \cO\) have type \code{RngPadElt}, and are defined by its residue class \[x + \pi^k \cO\] for some \(k \ge 0\). This is its \define{absolute precision}, which can be different for each element. If \(K=\QQ_p\) or \(K\) is semi-exact then \(k\) is unbounded, but if \(K\) is defined by an inexact polynomial, then \(k\) is bounded by the absolute precision of that polynomial.

If \(x \in \pi^k \cO\) then we say \(x\) is \define{weakly zero}. If \(x\) is zero then it is weakly zero, but the converse does not hold, if \(x\) is weakly zero we only know \(\val(x)\ge k\). If \(x-y\) is weakly zero, then we say \(x\) and \(y\) are \define{weakly equal}.

Field elements \(x \in K\) have type \code{FldPadElt}, and are defined by the pair \[(v, y) \in \ZZ \times \cO \quad\text{where}\quad x = \pi^v y.\] Clearly \(\val(x) \ge v\), and we refer to \(v\) as the \define{weak valuation} of \(x\).

Recalling that \(y\) is defined by \(y + \pi^k \cO\), then \(x\) is equivalently defined by \(\pi^v y + \pi^{v+k} \cO\). The \define{absolute precision} of \(x\) is \(v+k\) and the \define{relative precision} of \(x\) is \(k\).

This definition is normalized so that if \(k>0\) then \(y\in\cO^\times\). Hence \(x\) is weakly zero if and only if \(k=0\), and if \(x\) is not weakly zero then \(\val(x)=v\).

Additionally, \(v=\infty\) is allowed, in which case \(x=0\) is the \define{precise zero}. In this case, we set \(k=0\) so that the precise zero is also weakly zero.

\section{Our implementation}
\label{sec-exact}

\subsection{Core description}

All \define{\(p\)-adic objects} for which we provide an exact implementation (such as \(p\)-adic fields, \(p\)-adic integer rings, elements thereof, and polynomials defined over these) are represented in a common fashion, which we describe in this section.

Let \(X\) be such an object, then \(X\) is defined by three pieces of information: its \define{type}, its \define{dependencies} and its \define{kind}.

The type means its type in the Magma type system, including \code{FldXPad} for \(p\)-adic fields and \code{FldXPadElt} for their elements. All exact types have \code{XPad} in their name.

The dependencies is a list of other objects of any type (not necessarily exact \(p\)-adic objects) on which the value of \(X\) depends. For example, if \(X\) is an element of \(K\) equal to the integer \(x\in\ZZ\) then the dependencies of \(X\) may be \([K, x]\).

The kind describes how \(X\) may be computed given its dependencies. For example, there is a kind for \(p\)-adic field elements which are coerced from integers, as in the previous paragraph. Similarly, there is a kind for fields which are an unramified extension of some other field, or for a field element which is the sum of two other elements of the same field, or for a polynomial which is an irreducible factor of another polynomial.

Concretely, the kind is an integer which indexes into a static table of kind information for the type, fully describing the kind. Each entry contains: the number of dependencies of objects of this kind (or a wildcard value meaning any number is allowed), the types expected of each dependency, and a \code{GetApprox} function, whose usage we describe shortly.

The type, dependencies and kind define the object \(X\) by defining an infinite sequence \([\tilde X_1, \tilde X_2, \ldots, \tilde X_n, \ldots]\) of increasingly precise (inexact) approximations to \(X\). For example, the approximations of a \(p\)-adic field element of type \code{FldXPadElt} are inexact \(p\)-adic field elements of type \code{FldPadElt}.

The \(n\)th term of this sequence is defined and computed as follows. Let \([D^{(1)}, \ldots, D^{(d)}]\) be the dependencies of \(X\). For \(i=1,\ldots,d\), if \(D^{(i)}\) is an exact \(p\)-adic object, compute \(\tilde D^{(i)}_n\) recursively; otherwise define \(\tilde D^{(i)}_n := D^{(i)}\). Then \[\tilde X_n := \text{\code{GetApprox}}(n, [\tilde D^{(1)}_n, \ldots, \tilde D^{(d)}_n])\] where the \code{GetApprox} function is looked up using the type and kind of \(X\).

We describe \(n\) as the \define{epoch}, so that \(\tilde X_n\) is the \define{approximation of \(X\) at epoch \(n\)}. Then the above definition may be put as: the approximation of \(X\) at epoch \(n\) is the \code{GetApprox} function called with the epoch \(n\) and the list of approximations of its dependencies at epoch \(n\).

Each element keeps a cache \([\tilde X_1, \ldots, \tilde X_n]\) of the approximations computed so far. We require that precision is not lost: if \(n_1\le n_2\) then \(\tilde X_{n_2}\) is at least as precise as \(\tilde X_{n_1}\). Figure~\ref{fig-type} illustrates the relationships between defining information of \(X\).

If \(X\) is an exact \(p\)-adic structure, such as a ring of \(p\)-adic integers, then its elements are also exact \(p\)-adic objects. If \(x \in X\) is such an element, then we require that \(\tilde x_n \in \tilde X_n\) for all \(n\). That is, the approximation of an element of a structure must be an element of the approximation of the structure. In particular approximations of different elements of \(X\) at the same epoch belong to the same approximate structure.

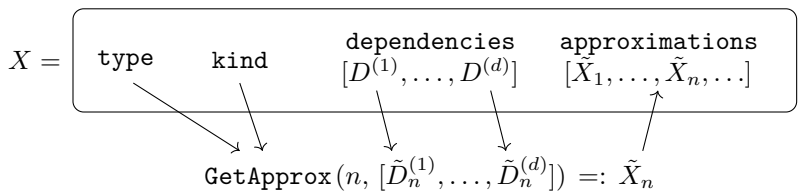
\begin{figure}
\centering
\begin{tikzpicture}[
  attr/.style={font=\ttfamily},
  box/.style={draw, inner sep=0.2cm, rounded corners},
  remember picture,
]
\node(X) at (-0.2,0) {\(X =\)};
\node(type) at (1,0) [attr] {type};
\node(kind) at (2.5,0) [attr] {kind};
\node(deps) at (5,0) [attr, text width=3cm, align=center] {dependencies \([D^{(1)},\ldots,D^{(d)}]\)};
\node(apps) at (8,0) [attr, text width=3cm, align=center] {approximations \([\tilde X_1, \ldots, \tilde X_n, \ldots]\)};
\node(box)[box, fit={(type) (kind) (deps) (apps)}] {};

\node(eq) at (5,-1.5) {\(\code{GetApprox}\,(n, \,[\tilde D_n^{(1)}, \ldots, \tilde D_n^{(d)}]) \,=:\, \tilde X_n\)};

\draw[->](1.1,-0.3)--(2.5,-1.2);
\draw[->](2.5,-0.3)--(2.8,-1.2);
\draw[->](4.3,-0.4)--(4.5,-1.2);
\draw[->](5.8,-0.4)--(6,-1.2);
\draw[->](7.8,-1.2)--(8,-0.4);
\end{tikzpicture}
\caption{Illustration of an exact \(p\)-adic object \(X\).}
\label{fig-type}
\end{figure}

\subsection{Particular types}

In this section, we describe the particular types currently provided by our implementation, namely rings and fields of \(p\)-adic numbers, their elements, and univariate polynomials over these.

We plan to support more aggregate types in the future, including multivariate polynomials, vectors and matrices. These can be modelled on the univariate polynomials.

\subsubsection{\texttt{FldXPad} and \texttt{RngXPad}: fields of \texorpdfstring{\(p\)}{p}-adic numbers and rings of \texorpdfstring{\(p\)}{p}-adic integers}

These have an almost identical implementation to each other. An approximation of a \code{FldXPad}, as returned by its \code{GetApprox} function, is a finite-precision \code{FldPad} whose precision increases with the epoch. Similarly an approximation of a \code{RngXPad} is a finite-precision \code{RngPad}.

There are three kinds of these fields and rings. Firstly there is the kind of prime \(p\)-adic fields and rings, namely \(\QQ_p\) and \(\ZZ_p\). The approximation at epoch \(n\) is the corresponding finite-precision ring or field of precision \(2^n\). See Section~\ref{sec-exp-prec} for the rationale behind this choice of precision.

Then there is the kind of unramified extensions and the kind of totally ramified extensions. These are defined by a base ring or field \(R\) and an inertial or Eisenstein polynomial \(f(x) \in R[x]\) (itself an exact \(p\)-adic polynomial of type \code{RngUPolXPad}). If \(S\) is the extension defined by \(f\), then its approximation \(\tilde S_n\) is the extension of \(\tilde R_n\) defined by \(\tilde f_n\). The defining polynomial is cached by the field or ring, so that the base field, degree and other information about the extension can be deduced.

\subsubsection{\texttt{FldXPadElt} and \texttt{RngXPadElt}: \texorpdfstring{\(p\)}{p}-adic numbers and integers}

An element of an exact \(p\)-adic number field of type \code{FldXPad} has type \code{FldXPadElt}, and similarly for integers. As with the rings and fields themselves, the implementation of their elements is almost identical to each other. Since the approximation of an element of a structure must be an element of the approximation of the structure, an approximation of a \code{FldXPadElt} is an element of a \code{FldPad}, which is a finite-precision \(p\)-adic number of type \code{FldPadElt} as described in Section~\ref{sec-inexact}.

For example the sequence of approximations of \(1 \in \ZZ_3\) might be \[[1+3^2\ZZ_3, 1+3^4\ZZ_3, 1+3^8\ZZ_3, \ldots].\]

There are many kinds of \(p\)-adic numbers. For example, there are kinds for numbers coerced from integers or rationals, or coerced from a base field, or lifted from the residue class field. There are kinds for special elements of the parent ring or field, such as a uniformizing element or generator. There are kinds for numbers arising from arithmetic operations on other numbers, such as addition, multiplication, division and powering. There are also kinds for numbers obtained from more complex routines, such as roots of polynomials obtained via Hensel lifting.

The implementation of the \code{GetApprox} function for many of these kinds is very simple. For example if \(K\) is a \(p\)-adic field and \(a\in\ZZ\) and \(x=\code{Coerce}(K,a)\) is the element of \(K\) equal to \(a\) then \(x\) has dependencies \([K,a]\) and \[\tilde x_n := \code{GetApprox}(n,[\tilde K_n, a]) = \code{Coerce}(\tilde K_n, a).\] Similarly if \(x,y\in K\) and \(z:=x+y\in K\) then \[\tilde z_n := \code{GetApprox}(n,[\tilde x_n, \tilde y_n]) = \tilde x_n + \tilde y_n.\] One must be careful that the \code{GetApprox} function is well-defined at all epochs, for example division requires some extra care (see Section~\ref{sec-min-epoch}).

\subsubsection{\texttt{RngUPolXPad} and \texttt{RngUPolXPadElt}: Univariate polynomial rings and their elements}
\label{sec-rngupolxpad}

We provide the type \code{RngUPolXPad} to represent the ring of polynomials over any exact \(p\)-adic ring, such as a \(p\)-adic field. Its elements have type \code{RngUPolXPadElt}, and both types are exact \(p\)-adic objects, in the established sense that they are defined by a list of dependencies and a kind.

Note that Magma already provides types \code{RngUPol} and \code{RngUPolElt} for univariate polynomials over arbitrary rings. See Section~\ref{sec-aggregates} for why we define a new representation.

There is only one kind of univariate polynomial ring, defined by its base ring. It has one dependency, the base ring, and its approximation at epoch \(n\) is the univariate polynomial ring (of type \code{RngUPol}) over the approximation at epoch \(n\) of its base ring.

\subsection{Further details}

\subsubsection{Minimum epoch}
\label{sec-min-epoch}
\begin{example}
Suppose \(x, y \in \QQ_p\) and we wish to compute \(z = x/y\). Of course, we first check that \(y\) is non-zero, which we do by inspecting each approximation \(\tilde y_n\) of \(y\) in turn until we find one which is not weakly zero. Now suppose \(z\) has dependencies \([x,y]\) and approximations \[\tilde z_n := \code{GetApprox}(n,[\tilde x_n, \tilde y_n]) = \tilde x_n / \tilde y_n.\] Even though \(y\) is non-zero, it might be the case that its first approximation is weakly zero, namely \(\tilde y_1 = 0 + p^v \ZZ_p\), which results in a division-by-zero error when computing \(\tilde x_1/\tilde y_1\).
\end{example}

To avoid this problem, each object \(X\) has an additional piece of information called \code{min\_epoch}, which is the smallest \(n\) for which the corresponding \code{GetApprox} function may be called. That is, we restrict the earlier definition: \[\tilde X_n := \code{GetApprox}(n,[\tilde D_n^{(1)},\ldots,\tilde D_n^{(d)}])\quad\text{for }n\ge\code{min\_epoch}.\] We obtain the approximations at smaller epochs by limiting the precision of \(\tilde X_{\code{min\_epoch}}\) appropriately. In particular, in the case where \(X\) is an element of a structure \(S\), we define \[\tilde X_n := \code{Coerce}(\tilde S_n, \tilde X_{\code{min\_epoch}})\quad\text{for }n<\code{min\_epoch}.\]

In the division example, we set the \code{min\_epoch} to be the smallest epoch \(n\) such that \(\tilde y_n\) is not weakly zero. Since we require that precision must not be lost as \(n\) is increased, it is guaranteed that all subsequent approximations are also not weakly zero.

Another example of using \code{min\_epoch} is ensuring that the leading coefficient of a polynomial is non-zero when computing its discriminant. Another example is ensuring that an approximate root is sufficiently precise to guarantee that Hensel lifting it to a root of a polynomial will succeed.

\subsubsection{Infinite-precision approximations}
\label{sec-ipa}

Recall that if \(S/R\) is an extension of \(p\)-adic number fields or integer rings, defined by the irreducible polynomial \(f(x)\in R[x]\), then its approximation \(\tilde S_n\) is the extension of \(\tilde R_n\) defined by \(\tilde f_n\).

In Magma however, if we were to directly generate these extensions \(\tilde S_n/\tilde R_n\) separately for each epoch, they would not be considered to be related in any way. In particular, coercion between \(\tilde S_n\) at different epochs \(n\) would not be possible.

Instead, we attach to \(S\) its \define{infinite-precision approximation}, denoted \(\tilde S_\infty\), which is a semi-exact representation of \(S\) as described in Section~\ref{sec-inexact}. Its defining map \(m:\NN\to \tilde R_\infty[x]\) returns a sufficiently precise approximation of \(f(x)\).

Then \(S\) has dependencies \([\tilde S_\infty, f]\) and its approximations are actually computed as \[\tilde S_n := \code{GetApprox}(n, [\tilde S_\infty, \tilde f_n]) := \code{ChangePrecision}(\tilde S_\infty, \code{AbsolutePrecision}(\tilde f_n)).\] This definition ensures that the approximation of \(f\) returned by \(m\) is already available.


\subsubsection{Main algorithm}

For clarity, we now present the main algorithm for computing an approximation at epoch \(n\), including using the cache of approximations (which afterwards contains all approximations up to epoch \(n\)), checking the generated approximations are valid (using \code{IsValidApproximation}, which for example checks consistency as in Section~\ref{sec-no-reuse}), and accounting for the minimum epoch.

\begin{algorithm}[\(\code{Approximation}(X,n)\)]
\label{alg:approximation}
Given an exact \(p\)-adic object \(X\) and an epoch \(n\in\NN\), returns the approximation \(\tilde X_n\) of \(X\) at epoch \(n\).
\begin{algorithmic}[1]
\State \(A := \code{ApproximationsCache}(X)\)
\If{\(n > \#A\)}
    \State \(n_{\text{min}} := \code{MinEpoch}(X)\)
    \State \(D := \code{Dependencies}(X)\)
    \State \(\code{GetApprox} := \code{GetApproxFunction}(\code{Type}(X), \code{Kind}(X))\)
    \LineComment{For each epoch to do above the minimum epoch}
    \For{\(i := \max(\#A+1, n_{\text{min}}),\ldots,\max(n,n_{\text{min}})\)}
        \LineComment{Create and check an approximation from the dependencies}
        \State \(\tilde D := [\code{Approximation}(D[j],i) \,:\, j:=1,\ldots,\#D]\)
        \State \(\tilde X := \code{GetApprox}(i,\tilde D)\)
        \State \Assert \(\code{IsValidApproximation}(\tilde X, X, i)\)
        \LineComment{If required, create approximations up to the minimum epoch}
        \If{\(i = n_{\text{min}}\) \textbf{and} \(i > \#A+1\)}
            \State \(S := \code{Parent}(X)\)
            \For{\(j := \#A+1, \ldots, i-1\)}
                \State \(\tilde S := \code{Approximation}(S, j)\)
                \State \(\tilde X' := \code{Coerce}(\tilde S, \tilde X)\)
                \State \Assert \(\code{IsValidApproximation}(\tilde X', X, j)\)
                \State \(\code{Append}(A, \tilde X')\)
            \EndFor
        \EndIf
        \LineComment{Keep this approximation}
        \State \(\code{Append}(A, \tilde X)\)
    \EndFor
\EndIf
\State \Return \(A[n]\)
\end{algorithmic}
\end{algorithm}

\subsubsection{User-defined kinds}
\label{sec-user-kind}

We overload the \code{elt<>} syntax to allow the creation of elements with user-defined kinds in Magma. That is, if \(S\) is an exact \(p\)-adic structure then \[X := \code{elt<\(S\) | ga, \([D^{(1)},\ldots,D^{(d)}]\)>}\] is an element of \(S\) whose approximations are \[\tilde X_n := \code{ga}([\tilde D^{(1)}_n,\ldots,\tilde D^{(d)}_n]).\]

The user get-approx function \code{ga} may be viewed as defining a user-defined kind of element of \(S\).

In fact, all such elements have the same kind. The dependencies of \(X\) are actually \([\code{ga},D^{(1)},\ldots,D^{(d)}]\) and \[\code{GetApprox}(n, [\code{ga},\tilde D^{(1)}_n,\ldots,\tilde D^{(d)}_n]) := \code{ga}([\tilde D^{(1)}_n,\ldots,\tilde D^{(d)}_n]).\]

This construction is particularly useful when \(X\) is the output of some complex algorithm which is simpler to implement in Magma than in the C back-end. For example, our polynomial factorization and Hensel lifting routines use this (see Section~\ref{xp-sec-hensel}).

\subsubsection{Optimizing dependencies}
\label{sec-optimize}

Suppose \(z = f(x^{(1)},\ldots,x^{(d)})\) where \(f\) is some complicated function in exact \(p\)-adic objects \(\cX=[x^{(1)},\ldots,x^{(d)}]\). In particular, while \(z\) recursively depends ultimately on \(\cX\), there are a lot of intermediate values in the graph of dependencies between them.

We now present an optimization which produces a new exact \(p\)-adic object \(z'\) which is equal to \(z\) but whose dependencies are precisely \(\cX\). That is, we eliminate all the intermediate values. The \code{GetApprox} function for \(z'\) will execute a straight-line program to compute an approximation of \(z\) from approximations of \(\cX\).

\paragraph{What is optimized}
The intermediate values have been eliminated, so that the dependency graph is substantially simpler. This means less time will be spent in dependency-tracking code. In particular, to compute an approximation of \(z'\) we have to evaluate a straight-line program, instead of navigating the graph of dependencies of \(z\).

Furthermore, approximations of the intermediate values are computed when needed and then discarded, reducing persistent memory usage.

See Section~\ref{sec-timings} for the practical effect of this optimization.

\paragraph{Down-sides}
The time spent constructing \(z'\) from \(z\) may not be negligible. The \code{min\_epoch} of \(z'\) must be the maximum of the \code{min\_epoch} of any intermediate value between \(z\) and \(\cX\).

\paragraph{Construction}
Given exact \(p\)-adic object \(z\) and a set of exact \(p\)-adic objects \(\cX=[x^{(1)},\ldots,x^{(d)}]\), we construct the optimized \(z'\) as follows.

First, compute the directed acyclic graph of dependencies whose source is \(z\) and whose sinks include \(\cX\) and any non-exact dependencies.

Topologically sort this graph, yielding dependencies \([y^{(1)},\ldots,y^{(m)}]\), so that if \(y^{(i)}\) is a dependency of \(y^{(j)}\) then \(i<j\). We can ensure that \([y^{(1)},\ldots,y^{(r)}]\) are the inexact dependencies, \([y^{(r+1)},\ldots,y^{(r+d)}] = \cX\), and \(y^{(m)} = z\).

For each \(i=r+d+1,\ldots,m\), find the type \(T_i\) and kind \(K_i\) of \(y^{(i)}\), and the indices \(J_i=[j_1,\ldots,j_{d_i}]\) of the dependencies \([y^{(j)}]_{j\in J_i}\) of \(y^{(i)}\). Compute the sequence \(P=[(T_i,K_i,J_i)]_{i=r+d+1}^m\), which is the program.

The dependencies of \(z'\) are \([P, y^{(1)},\ldots, y^{(r+d)}]\), that is, the program and any required inputs. The \code{GetApprox} function receives approximations \([P, \tilde y^{(1)}_n,\ldots, \tilde y^{(r+d)}_n]\), and computes \(\tilde y^{(i)}_n\) successively for \(i=r+d+1,\ldots,m\) using the program. In particular we find the \code{GetApprox} function for the type \(T_i\) and kind \(K_i\) of \(y^{(i)}\), and call it with the approximations \([\tilde y^{(j)}_n]_{j\in J_i}\) which have already been computed. The final approximation computed is \(\tilde y^{(m)}_n = \tilde z_n\), which is returned.

\section{Design rationale}
\label{sec-rationale}

\subsection{Exponentially increasing precision}
\label{sec-exp-prec}

Recall that we define the precision of \(\QQ_p\) at epoch \(n\) to be \(2^n\). All computations over \(\QQ_p\) and its extensions occur to this precision by default. In this section we justify this choice of precision.

Suppose that an inexact computation at precision \(k\) has some cost (e.g. run-time or memory usage) \(C(k)\propto k^\alpha\). Suppose that at epoch \(n\) we work to precision \(b^n\) for some \(b>1\), and therefore using exact \(p\)-adics to achieve at least the same precision has cost \[C^*(k) := \sum_{i=1}^n C(b^i) \propto b^\alpha \frac{b^{\alpha n}-1}{b^\alpha-1}\] where \(b^{n-1}\le k\le b^n\), since this will work at each precision \(b,b^2,\ldots,b^n\) successively.

Hence \(b^n\le bk\) and the relative overhead of using exact versus inexact \(p\)-adics is \[R(k) := \frac{C^*(k)}{C(k)} \le \frac{b^{\alpha+1}}{b^\alpha-1} =: r(\alpha,b).\]

Setting \(\tfrac{\partial}{\partial b} r(\alpha,b)=0\), we find that \(r(\alpha,b)\) is minimized by \(b=b^*:=(1+\alpha)^{\alpha^{-1}}\) and \[r^*(\alpha) := r(\alpha, b^*) = \frac{(1+\alpha)^{1+\alpha^{-1}}}{\alpha}.\] In particular, for linear-cost operations (\(\alpha=1\)) this optimum is \(b^*=2\) and \(r^*=4\). That is, our choice of \(b=2\) is optimal with respect to relative cost of linear operations between exact and inexact arithmetic.

Furthermore \(\tfrac{\partial}{\partial\alpha} r(\alpha,b)<0\), so we deduce that \(R(k)\le4\) when \(\alpha\ge1\). That is, the relative overhead is at most 4 for any super-linear computation.

We can also ask: since \(b=2\) is only optimal for \(\alpha=1\), how much further could we have reduced \(R(k)\) for a specific \(\alpha>1\)? Since \(2<e\), we have \(2^\alpha\ge1+\alpha\) and we deduce \[\alpha2^\alpha\le(2^\alpha-1)(1+\alpha)\le(2^\alpha-1)(1+\alpha)^{1+\alpha^{-1}}\] which rearranges to \[r(\alpha,2)\le 2\cdot r^*(\alpha).\] That is, with our non-optimal choice of \(b=2\) for \(\alpha>1\), we are still within a factor of 2 of the optimal overhead. This ratio grows quite slowly, so even for cubic-cost algorithms (\(\alpha=3\)) then \(r(\alpha,2)\le1.08\cdot r^*(\alpha)\). That is, optimizing instead for cubic-cost operations would save only around \(10\%\) compared to using \(b=2\).

We emphasize that this factor of 4 overhead is in the worst-case scenario that the inexact computation is performed exactly once, at the correct precision, which must have been known in advance. In the more common scenario that the correct precision was not known in advance, the inexact computation would necessarily be repeated at increasing precisions just as for the exact computation, in which case the overhead factor is 1 for linear operations, and no more than 2 for super-linear operations.

\subsection{Epoch versus precision}
\label{sec-epoch-vs-prec}

Perhaps a more obvious scheme for lazily computed \(p\)-adics would be one which deals with precisions of elements directly, instead of using the epoch. Such a scheme is implemented in the prototype Magma package \code{ExactpAdics} \citep{exactpadicscode} and works essentially as follows. Each element \(X\) is still defined by its type, dependencies and kind, but now the kind stores slightly different information. Firstly, there is a function \[[k_1,\ldots,k_d] := \code{GetReqPrec}(k, [D^{(1)},\ldots,D^{(d)}])\] which returns a precision for each dependency. Secondly, there is a function \[\tilde X_{[k]} := \code{GetApprox}(k, [\tilde D^{(1)}_{[k_1]}, \ldots, \tilde D^{(d)}_{[k_d]}])\] which returns the approximation of \(X\) to precision \(k\), given approximations of its dependencies at precisions \([k_1,\ldots,k_d]\). With this pair of functions, we can compute \(\tilde X_{[k]}\) by computing \([k_1,\ldots,k_d]\) using \code{GetReqPrec}, then recursively computing each \(\tilde D^{(i)}_{[k_i]}\), then calling \code{GetApprox}.

We now discuss the relative merits of these schemes.

\subsubsection{Precision optimality}

The main advantage of the precision-based scheme is that one can compute an approximation \(\tilde X_{[k]}\) to some required precision using as little precision as possible in all its dependent computations.

It is rare, however, that one needs to know an approximation to some specific precision. It is more common that one is interested in some other property of the object, such as its valuation, for which any sufficiently large approximation will do. In such cases, it is likely that one will compute the approximation at successively higher precisions, such as by repeatedly doubling the precision, which is essentially what the epoch scheme does anyway.

Nevertheless, there are corner cases where this could matter. For example, suppose we compute \(z := f(x, y)\) where \(x\) is very expensive to compute, and where \(y\) is cheap to compute but loses a lot of precision. Under the epoch scheme, since \(y\) loses a lot of precision, so will \(z\), and therefore to achieve a particular precision we must compute \(\tilde z_n\) at some relatively high epoch \(n\), which is expensive because computing \(\tilde x_n\) at this epoch is expensive. That is, \(\tilde x_n\) is being computed to an unnecessarily high precision, a problem avoided by the precision scheme. Such corner cases seem rare.


\subsubsection{Complexity of implementation and performance}

The \code{GetApprox} function in both schemes will be implemented almost identically. On the other hand, the precision scheme additionally requires the \code{GetReqPrec} function. This function is used to do a \define{backwards pass} to determine the precisions required of all recursive dependencies, followed by a \define{forward pass} using \code{GetApprox} to actually compute the approximations. The epoch scheme only does a forward pass.

In simple cases, such as addition or multiplication of \(p\)-adic numbers, \code{GetReqPrec} is simple to implement. However, something just slightly more complex such as polynomial multiplication gets very unwieldy to implement: one has to find all the places where each coefficient contributes to the result, find what it is multiplied with, find the precision required, and take the maximum over all such places. This is also a significant amount of extra computation.

Worse, in some cases, it is practically impossible to determine the precision required in advance. This necessitates further complexity in the implementation: we must allow \code{GetReqPrec} to return a ``best guess'' at the precision required, and allow \code{GetApprox} to return a null response, indicating that the precision required was not enough. At this point, the implementation increases the guessed precision by further calls to \code{GetReqPrec} until \code{GetApprox} succeeds. This means that we cannot easily bound the number of calls to \code{GetApprox} for any particular object, and so we could spend arbitrarily long doing dependency tracking. Additionally, the benefit of precision optimality is nullified in this case. Contrast this with the epoch scheme, where \code{GetApprox} is called precisely once per epoch per object. Note also that if such a null response occurs, the forward pass must be stopped and a partial backwards pass must occur again.

Observe that in the precision scheme, during the backwards pass the same dependency \(D\) might occur multiple times, possibly with different precisions \(k_1, k_2, \ldots\). In a naive implementation, during the forwards pass the precision of \(D\) would be increased multiple times, to \(k_1\) then to \(k_2\) and so on. To avoid this, we can keep track of which dependencies appear multiple times, and take the supremum of their required dependencies. To do this efficiently requires explicitly representing the directed acyclic graph of dependencies. In contrast, this is avoided in the epoch scheme by definition: even if the same dependency occurs multiple times, it is always required at the same epoch \(n\). 


\subsubsection{Representation of precision}

In the precision scheme, we need to define what we mean by precision. The absolute or relative precision of a \(p\)-adic number is simply an integer or \(\infty\). But for aggregate structures such as polynomials, there is a choice to make for how to define them.

One option would be to define the precision of a polynomial \(f(x)=\sum_i f_ix^i\) to be the minimum precision of its coefficients \[\code{precision}(f) := \min_i \code{precision}(f_i)\] which is itself an integer, but this loses information since some coefficients may have a higher precision. In particular, this could mean that a higher precision gets used than was necessary, nullifying the benefits of precision optimality.

Another option is to define the precision to be the sequence of precisions of its coefficients \[\code{precision}(f) := (\code{precision}(f_0), \code{precision}(f_1), \ldots).\] This does not lose information, but now the precision is represented by an object of similar complexity to the polynomial itself, and \code{GetReqPrec} has even more to compute.

\subsection{Do not re-use previous approximations}
\label{sec-no-reuse}

One might expect that we could use the approximation \(\tilde X_{n-1}\) in order to compute \(\tilde X_n\) more quickly, since by definition they agree up to their common precision and only the difference needs to be computed. There are a number of reasons we do not do this.

Firstly, this would essentially require re-implementing all inexact \(p\)-adic operations from scratch in order to use this extra information, instead of the present scheme which simply directly uses the inexact \(p\)-adic operations as they are.

Secondly, the redundancy allows for error-checking, namely that \(\tilde X_{n-1}\) and \(\tilde X_n\) are weakly equal as expected. In practice this is a very useful way to catch precision-related errors coming from the underlying inexact arithmetic. For example it is not an uncommon implementation error that the trailing few \(p\)-adic coefficients of a ramified field element are incorrect after reducing modulo the defining Eisenstein polynomial.

Thirdly, by similar arguments as in Section~\ref{sec-exp-prec}, this would not actually save much computation time: since precision is increasing exponentially, the proportion of time spent repeating previous computations is bounded less than 1. See Section~\ref{sec-timings}.

\subsection{Exact aggregate structures}
\label{sec-aggregates}

We earlier described our implementation of polynomials over exact \(p\)-adic rings, and stated our intention to define other structures such as vector spaces, despite the fact that such structures already exist generically in Magma over any base ring. We now justify these new representations.

Consider the multiplication \(h(x):=f(x)g(x)\in K[x]\) of two univariate polynomials of degree \(d\) over an exact \(p\)-adic field \(K\), which entails multiplying around \(d^2\) pairs of coefficients. Using the generic polynomial type \code{RngUPolElt}, which stores the coefficients individually, these \(d^2\) intermediate values exist forever as dependencies of the coefficients of \(h(x)\). In contrast, using our \code{RngUPolXPadElt} then \(h\) depends only on \(f\) and \(g\). This shrinking of the dependency graph not only reduces the time spent tracking dependencies, but also reduces the memory usage since the intermediate variables in the computation of \(\tilde h_n\) are deleted once it has been computed.

Now consider another example, let \(f(x)\in K[x]\) be a polynomial and let \(g(x)\) be one of its irreducible factors, such as returned from a factorization routine. The approximations of \(g\) are computed using Hensel lifting, and in particular an approximation of \(g\) can only be computed in its entirety. That is, it is not known how to compute only a single coefficient of \(g\) to high precision faster than computing all coefficients. Hence it is natural to represent \(g\) in its entirety, as \code{RngUPolXPadElt} does, instead of by a sequence of coefficients. If \(g\) were represented naively as a generic \code{RngUPolElt}, then in order to compute an approximation of one of its coefficients, we would have to compute an approximation of the whole polynomial via Hensel lifting and extract the coefficient, and hence computing an approximation to all coefficients of \(g\) would repeat the same Hensel lifting step \(\deg g\) times.

As a compromise between usability and efficiency, we still allow the user to use the generic \code{RngUPolElt} type, but implicitly convert this to \code{RngUPolXPadElt} and back for certain operations.

\section{Performance comparison}
\label{sec-timings}

In this section, we compare the performance of our exact \(p\)-adic arithmetic versus the equivalent inexact \(p\)-adic arithmetic in Magma. Since our exact arithmetic uses the inexact arithmetic internally, these comparisons measure the other overheads in the exact arithmetic, such as from recursively resolving dependencies.

For this reason, we are not making comparisons to \(p\)-adic systems in other computer algebra systems, since this would be comparing the performance of Magma's inexact arithmetic, which is beyond the scope of this article.

\paragraph{Experiment 1}

We compute \(y = \sum_{i=1}^N x_i\) where \(x_1=1\in\QQ_2\), \(x_2=2\in\QQ_2\) and \(x_i=x_{j_i}+x_{k_i}\) for \(i=2,\ldots,N\) where \(1\le j_i,k_i<i\) are chosen at random. Using our exact arithmetic, we compute approximations of \(y\) at epochs \(n=1,\ldots,16\). Using the inexact arithmetic, we also compute \(y\) to absolute precisions \(k=2^1,2^2,\ldots,2^{16}\), which emulates the computations being performed in the exact case. The same random choices are used for each computation.

Note that the definition of \(y\) is designed to be something which is quick to compute, but which has a complex set of dependencies, and therefore should highlight any overheads due to dependency tracking.

We repeat these computations 50 times with different random choices, and report the average times in Table~\ref{tbl-timings-1} along with standard deviations. For the exact arithmetic, we also report separately the time to construct \(y\) and the time to compute its approximations. We also report the time just to compute the final approximation, at epoch 16 or precision \(2^{16}\).

Note that we repeat this experiment with three different modes of use of exact arithmetic: the default behaviour; or the default behaviour but with consistency checking of each approximation disabled; or with the optimization from Section~\ref{sec-optimize} enabled.

\begin{table}
\centering
\begin{tabular}{l rcrcrr}
\hline
 & \multicolumn{6}{c}{Mean time (core seconds)} \\
Mode & Total \(\pm\sigma\) & \multicolumn{2}{r}{Constr.} & \multicolumn{2}{r}{Approx.} & Final \\
\hline
\(N=10,000\) \\
Inexact & \(0.25\pm 0.01\) & & & & & 0.02 \\
Exact (default) & \(0.31\pm 0.02\) &=& 0.01 &+& 0.30 & 0.02 \\
Exact (no checks) & \(0.18\pm 0.01\) &=& 0.01 &+& 0.17 & 0.01 \\
Exact (optimized) & \(0.14\pm 0.01\) &=& 0.02 &+& 0.12 & 0.01 \\
\hline
\(N=100,000\) \\
Inexact & \(3.01 \pm 0.02\) & & & & & 0.19 \\
Exact (default) & \(4.73\pm 0.16\) &=& 0.16 &+& 4.57 & 0.29 \\
Exact (no checks) & \(3.13\pm 0.15\) &=& 0.16 &+& 2.97 & 0.18 \\
Exact (optimized) & \(2.20\pm 0.02\) &=& 0.30 &+& 1.90 & 0.12 \\
\hline
\end{tabular}
\caption{Timings for Experiment 1 over different arithmetic modes.}
\label{tbl-timings-1}
\end{table}

Observe that the exact arithmetic is at most around \(1.60\) times slower than the equivalent inexact arithmetic, reducing to about \(1.05\) times slower with consistency checking disabled, implying that most of the overheads are in consistency checking.

With the optimization enabled, the computation is actually about \(1.35\) times faster, even though the time to construct the element has doubled. Note that the optimization by necessity does not perform consistency checking on any intermediate computations, which explains part but not all of the speed-up. It is likely that the inexact arithmetic has some overheads coming from the Magma interpreter, which is invoked much less in the exact arithmetic.

We conclude that there is not a large overhead in using our exact \(p\)-adic system compared to manually repeating a computation at successively increasing precisions. What overhead there is can be eliminated or even reversed using the optimization.

\paragraph{Experiment 2}

We now repeat the previous experiment, except that we define \(x_1=1/3\), \(x_2=1/5\) and go up to epoch \(n=24\). The values of \(x_i\) and \(y\) are no longer small integers, and therefore require more significant computation. Given also the high epoch, we expect overheads to now be negligible in comparison to the main computation. In the notation of Section~\ref{sec-exp-prec}, in this experiment we expect \(C(k)\propto k=2^n\) (i.e. \(\alpha=1\)) whereas in the previous experiment we have \(C(k)\) being essentially a constant (i.e. \(\alpha=0\)).

\begin{table}
\centering
\begin{tabular}{l rcrcrr}
\hline
 & \multicolumn{6}{c}{Mean time (core seconds)} \\
Mode & Total \(\pm\sigma\) & \multicolumn{2}{r}{Constr.} & \multicolumn{2}{r}{Approx.} & Final \\
\hline
\(N=500\) \\
Inexact & \(1.72\pm0.05\) & & & & & \(0.88\) \\
Exact (default) & \(2.05\pm0.07\) &=& 0.00 &+& 2.05 & 1.05 \\
Exact (no checks) & \(1.75\pm0.05\) &=& 0.00 &+& 1.75 & 0.90 \\
Exact (optimized) & \(1.71\pm0.05\) &=& 0.00 &+& 1.71 & 0.89 \\
\hline
\end{tabular}
\caption{Timings for Experiment 2 over different arithmetic modes.}
\label{tbl-timings-2}
\end{table}

The timings from this experiment are in Table~\ref{tbl-timings-2}. Aside from a factor of \(1.2\) coming from consistency checking in the default mode, the timings are all very similar, up to variance. Hence our implementation introduces negligible other overheads in this case.

Also observe that in this case approximately half of the total time is due to the final approximation, as one expects due to the exponentially increasing precision, justifying Section~\ref{sec-no-reuse}.

\paragraph{Experiment 3}

In this experiment, we compare performance of our polynomial arithmetic with equivalent inexact computations in both Magma and Sage.

Defining
\begin{align*}
f_{d,0}(x) &= x^d - 2^{10d+1} \\
f_{d,u}(x) &= f_{d,0}(x-u) \\
g_{2d,u}(x) &= f_{d,u}(x)f_{d,u+2^{11}}(x)
\end{align*}
then \(g_{2d,u}(x)\) has two irreducible factors over \(\QQ_2\) of degree \(d\), with roots \(u+2^{10}\sqrt[d]2\) and \(u+2^{10}(2+\sqrt[d]2)\), which are all close to \(u\). This is designed to be difficult to factorize.

Using our exact arithmetic, we factorize \(f_{16,u}(x)\) over \(\QQ_2\), then obtain approximations of its factors at each epoch \(n=1,\ldots,16\). Using inexact arithmetic, we factorize \(f_{16,u}(x)\) to precision \(2^n\) for \(n=1,\ldots,16\). We also repeat the latter using Sage. We repeat all of this but now factorizing \(g_{16,u}(x)\).

These experiments are repeated for 10 random choices of units \(u\in\ZZ_2^\times\), with average timings given in Table~\ref{tbl-timings-3}. The ``Exact'' column is cumulative: it gives the total time to compute the given epoch and all lower epochs.


\begin{table}
\centering
\begin{tabular}{rrrr}
    \hline
    & \multicolumn{3}{c}{Mean time (core seconds)} \\
    \(n\) & Exact & Inexact & Sage \\
    \hline
     8 & 0.02 & 0.02 &  --- \\
     9 & 0.02 & 0.02 &  --- \\
    10 & 0.02 & 0.03 &  --- \\
    11 & 0.02 & 0.05 &  --- \\
    12 & 0.02 & 0.07 & 0.02 \\
    13 & 0.02 & 0.23 & 0.04 \\
    14 & 0.02 & 0.84 & 0.11 \\
    15 & 0.02 & 3.27 & 0.25 \\
    16 & 0.02 & 12.57 & 0.68 \\
    \hline
\end{tabular}
\hspace{2em}
\begin{tabular}{rrrr}
    \hline
    & \multicolumn{3}{c}{Mean time (core seconds)} \\
    \(n\) & Exact & Inexact & Sage \\
    \hline
     8 & 0.02 &  --- &  --- \\
     9 & 0.03 &  --- &  --- \\
    10 & 0.05 &  --- &  --- \\
    11 & 0.08 & 0.19 &  --- \\
    12 & 0.16 & 0.37 & 0.08 \\
    13 & 0.46 & 1.07 & 0.15 \\
    14 & 1.40 & 3.82 & 0.34 \\
    15 & 4.36 & 14.61 & 0.90 \\
    16 & 13.06 & 55.85 & 2.31 \\
\hline
\end{tabular}
\caption{Timings for Experiment 3. Left: irreducible \(f_{16,u}(x)\). Right: reducible \(g_{16,u}(x)\). Blank entries show failures.}
\label{tbl-timings-3}
\end{table}

In all examples, the exact computation required epoch \(n=8\) to compute the initial factorization. Given that the roots are indistinguishable modulo \(2^{10}\), distinguishable modulo \(2^{12}\), and the degree is 16, precision \(10\times16>2^7\) is necessary and \(12\times16\le2^8\) is sufficient, and so \(n=8\) is optimal. The inexact computations typically required higher precision to succeed, with the exeption of Magma in the irreducible \(f_{16,u}(x)\) case. However, in this case the Magma implementation also incorrectly factorizes this into 16 linear factors for \(n<8\) instead of raising an error.

On the reducible \(g_{16,u}(x)\), we see that Sage, when it succeeds, is somewhat quicker than our exact implementation at the same precision, with the gap widening as the precision increases, likely due to a superior Hensel lifting implementation. On irreducible \(f_{16,u}(x)\) our exact implementation is quicker because one can trivially avoid Hensel lifting altogether.

On the other hand, the exact factorization succeeds at a much lower precision than the other implementations and obtains its initial factorization more quickly. This is beneficial in situations where the factorization is not required to high precision, such as in Galois group computations where only the degrees of the factors are needed.

Furthermore, the Sage implementation relies on PARI/GP \citep{parigp} and therefore currently does not support extensions of \(\QQ_p\). Magma supports arbitrary finite extensions of \(\QQ_p\).

\section{Some particular algorithms}
\label{sec-algorithms}

\subsection{Valuation}

We begin by demonstrating an algorithm to compute the valuation of a \(p\)-adic number \(x\). It works through each epoch \(n=1,2,\ldots\) in turn. If at epoch \(n\), the valuation is known from the approximation \(\tilde x_n\) (namely the relative precision is positive, or \(\tilde x_n\) is the precise zero) then the weak valuation is returned, since this is equal to the valuation. Otherwise we move on to the next epoch.

\begin{algorithm}[\code{Valuation}]
Given a \(p\)-adic number \(x\), returns its valuation.
\begin{algorithmic}[1]
\For{\(n=1,2,\ldots\)}
    \If{\(\code{ValuationIsKnownAtEpoch}(x,n)\)}
        \State \Return \(\code{WeakValuationAtEpoch}(x,n)\)
    \EndIf
\EndFor
\end{algorithmic}
\end{algorithm}

Note that if we wish to query if \(\val(x) \ge 0\) then one option would be to compute \(\code{Valuation}(x) \ge 0\). However, if the valuation is actually very much higher than 0 then computing \(\code{Valuation}(x)\) exactly is wasted effort, it suffices to stop as soon as its weak valuation is greater than 0. If in fact \(x=0\) then \(\code{Valuation}(x)\) might not terminate. Hence we provide:

\begin{algorithm}[\code{ValuationCmp}]
Given a \(p\)-adic number \(x\) and an integer \(v\), returns \(-1\) if \(\val(x)<v\), \(0\) if \(\val(x)=v\), or \(1\) if \(\val(x)>v\).
\begin{algorithmic}[1]
\For{n=1,2,\ldots}
    \If{\(\code{ValuationIsKnownAtEpoch}(x,n)\) or \newline\hspace*{1.2cm} \(\code{AbsolutePrecisionAtEpoch}(x,n)>v\)}
        \State \Return \(\code{Cmp}(\code{WeakValuationAtEpoch}(x, n), v)\)
    \EndIf
\EndFor
\end{algorithmic}
\end{algorithm}

Unlike \code{Valuation}, \code{ValuationCmp} is guaranteed to terminate.

\subsection{Newton polygon}
\label{sec-newton-polygon}

Recall the definition and a key property of the Newton polygon of a univariate polynomial.

\begin{definition}
If \(f(x) = \sum_{i=0}^d f^{(i)} x^i \in K[x]\) is a polynomial over a \(p\)-adic field \(K\), then its \define{Newton polygon} \(\Delta\) is the lower convex hull in \(\QQ \times \QQ\) of the points \((i, \val(f^{(i)}))\). It can also be interpreted as the graph of a function \(\Delta:[0,d] \to \QQ\). By definition, this function is continuous, convex and piece-wise linear. 
\end{definition}

\begin{lemma}
\label{lem-newtonpgon}
Given a face of \(\Delta\), i.e. a maximal line segment in \(\Delta\) from \((i_0,v_0)\) to \((i_1,v_1)\), then \(f\) has precisely \(w=i_1-i_0\) roots in \(K^{\mathrm{alg}}\) of valuation \(-s=-\tfrac{v_1-v_0}{w}\). Writing \(-s=h/e\) in lowest terms, it follows that \(K(r)/K\) has ramification degree a multiple of \(e\).
\end{lemma}

Similarly, one can deduce the residue class of \(r^e/\pi^h\) (which has valuation 0) from the residual polynomial corresponding to the face, and obtain bounds on the residue degree of \(K(r)/K\). Together, this is the key information used in root finding and polynomial factoring algorithms over \(K\), and so we provide algorithms to compute them.

Computing the Newton polygon requires some care. One method could be to simply compute \(\code{Valuation}(f^{(i)})\) for each coefficient, but unless the polygon has a vertex at \(i\) then this is redundant and might not terminate, for similar reasons as motivated \code{ValuationCmp} above.

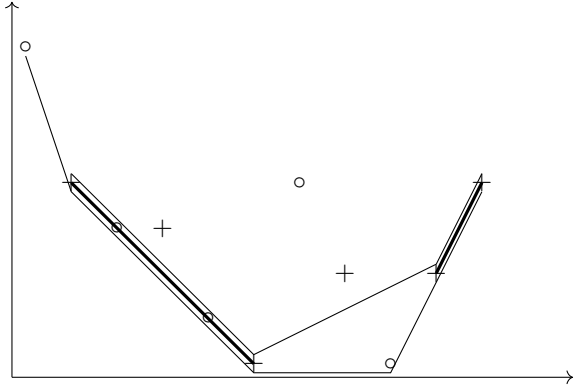
\begin{figure}
\centering
\begin{tikzpicture}[scale=0.6]
\draw[->](-0.3,-0.3)--(-0.3,8);
\draw[->](-0.3,-0.3)--(12,-0.3);
\foreach \x/\y in {0/7,2/3,4/1,6/4,8/0}
  \node(\x\y) at (\x,\y){\(\circ\)};
\foreach \x/\y in {1/4,3/3,5/0,7/2,9/2,10/4}
  \node(\x\y) at (\x,\y){\(+\)};
\draw([yshift=-0.2cm]07.center)--([yshift=-0.2cm]14.center)--([yshift=-0.2cm]50.center)--([yshift=-0.2cm]80.center)--([yshift=-0.2cm]104.center);
\draw[very thick](14.center)--(50.center);
\draw[very thick](92.center)--(104.center);
\draw([yshift=0.2cm]14.center)--([yshift=0.2cm]50.center)--([yshift=0.2cm]92.center)--([yshift=0.2cm]104.center);
\end{tikzpicture}
\caption[Computation of Newton polygon]{Computation of a section of a Newton polygon (heavy line) from lower and upper polygons. Circles indicate the weak valuations of weakly zero coefficients, crosses indicate valuations of non weakly zero coefficients. Observe that since each end of the leftmost piece of the Newton polygon is at a vertex of the lower polygon, then these must also be vertices of the Newton polygon; contrast with the rightmost piece, in which the face could extend further to the left.}
\label{fig-newtonpgon}
\end{figure}

Instead we run through epochs \(n=1,2,\ldots\), find the approximation \(\tilde f_n\) of \(f\) at that epoch, and compute two related polygons. The \define{lower polygon} of \(\tilde f_n\) is the lower convex hull of the points \((i,w_i)\) where \(w_i\) is the weak valuation of the \(i\)th coefficient \(\tilde f^{(i)}_n\). The \define{upper polygon} of \(\tilde f_n\) is the lower convex hull of the subset of the points \((i,w_i)\) such that the valuation of \(\tilde f^{(i)}_n\) is known, and therefore \(w_i=\val(f^{(i)})\). Since a weak valuation is a lower bound on the true valuation, the lower polygon lies below the Newton polygon. Since the upper polygon is defined using a subset of the points that defined the Newton polygon, it lies above the Newton polygon. Therefore, wherever the lower and upper polygon are equal, they are also equal to a section of the Newton polygon (see Figure~\ref{fig-newtonpgon}). Hence if the lower and upper polygons are equal, return the common polygon, otherwise move on to the next epoch.

The lower and upper polygons can also provide useful partial information about the Newton polygon. For example, in Figure~\ref{fig-newtonpgon} we can see that the Newton polygon has a vertex above 1, even though the face on \([0,1]\) is not known. This allows testing the Hensel condition in Lemma~\ref{lem-hensel-general}.

\subsection{Hensel's lemma for univariate root-finding}
\label{xp-sec-hensel}

Recall Hensel's classic lemma.

\begin{lemma}[Hensel]
\label{lem-hensel}
Suppose \(f(x) \in \cO[x]\), \(a \in \cO\) such that \(v(f(a)) \geq s > 0 = v(f'(a))\). Then there exists a unique \(b \in K\) such that \(f(b) = 0\) and \(v(a-b) \geq s\). More precisely, defining \(a' := a - f(a)/f'(a)\) then \(v(f(a')) \geq 2s\) and \(v(f'(a')) = 0\), so iterating \(a \mapsto a'\) then \(a \to b\).
\end{lemma}

We refer to the iteration process in Hensel's lemma as \define{Hensel lifting}. Hensel's lemma can be generalized to non-integral inputs:

\begin{lemma}[Non-integral Hensel]
\label{lem-hensel-general}
Suppose \(f(x) \in K[x]\), where \(K\) is a \(p\)-adic field, and \(a \in K\) such that among all roots \(b\) of \(f\), \(v(a-b)\) is maximized precisely once. Then iterating \(a \mapsto a - f(a)/f'(a)\) yields \(a \to b\).
\end{lemma}

\begin{proof}
The generalization is actually reducible to the original version.

Consider the polynomial \(f(x+a)\). Its roots are \(b-a\) where \(b\) is a root of \(f\), and so its Newton polygon measures the number of times each \(v(a-b)\) occurs. Hence the hypothesis is equivalent to saying that the first face of the Newton polygon of \(f(x+a)\) has width 1.

Suppose this is true, then in particular the first face has integral slope and so there exist \(j,k \in \ZZ\) so that \(g(x) := \pi^j f(\pi^k x + a)\) has integral coefficients, \(\val(g_0) > 0\) and \(\val(g_1) = 0\). Note that \(g_0 = g(0)\) and \(g_1 = g'(0)\) so the original version of Hensel's lemma applies to \(g\) and \(0\). By linearity, Hensel lifting on \(g\) is equivalent to Hensel lifting on \(f\).
\end{proof}

\begin{remark}
Krasner's lemma \cite[Ch. 7, Corr. 3 to Thm. 1.1]{Cassels} is a corollary of this form of Hensel's lemma.
\end{remark}

We provide an algorithm \code{IsHenselLiftable} which takes as input a polynomial \(f(x) \in K[x]\) and an element \(a \in K\) and returns true if this generalized version of Hensel's lemma can be applied to find a root \(b\) of \(f\) close to \(a\). If so, it also returns that root.

The algorithm proceeds by computing \(f(x+a)\) to sufficient precision to see if the first face of its Newton polygon has width 1 or not, as in Section~\ref{sec-newton-polygon}. If so, then the returned root has dependencies \(f\), \(a\) and some data from the Newton polygon, and its \code{GetApprox} function performs Hensel lifting to the required precision.

\begin{remark}
We also provide an implementation of an ``OM-algorithm'' (essentially that described in \cite[Ch. VI]{SinclairTh}) for finding the roots or irreducible factors of such a polynomial.
\end{remark}

\begin{remark}
Hensel's lemma can be generalized to systems of \(N\) multivariate polynomials in \(N\) variables, essentially by replacing the derivative \(f'(x)\) by the Jacobian matrix of derivatives, and there is a non-integral generalization too (e.g. \cite[Ch. IV, \S9.9]{DPhD}). Multivariate Hensel lifting can also be performed in the relaxed setting \citep{relaxedhensel}.

Furthermore, we may express a factorization \(f(x)=g(x)h(x)\) of a degree-\(N\) monic polynomial \(f\) as the solution to a system of \(N\) equations in \(N\) variables, where the equations are the coefficients of \(f-gh\) and the variables are the coefficients of \(g\) and \(h\). This yields a Hensel's lemma for univariate polynomial factorization, which in particular leads to an algorithm to factorize \(f\) according to the segments of its Newton polygon (e.g. \cite[Ch. IV, \S9.11]{DPhD}). The OM-algorithm mentioned above can be viewed as a generalization of this.
\end{remark}

\section*{Acknowledgements}
This work was partially supported by a grant from GCHQ.

\bibliographystyle{elsarticle-harv}

\begin{thebibliography}{18}
\expandafter\ifx\csname natexlab\endcsname\relax\def\natexlab#1{#1}\fi
\expandafter\ifx\csname url\endcsname\relax
  \def\url#1{\texttt{#1}}\fi
\expandafter\ifx\csname urlprefix\endcsname\relax\def\urlprefix{URL }\fi

\bibitem[{Berthomieu and Lebreton(2012)}]{relaxedhensel}
Berthomieu, J., Lebreton, R., 2012. Relaxed {$p$}-adic {H}ensel lifting for
  algebraic systems. In: I{SSAC} 2012---{P}roceedings of the 37th
  {I}nternational {S}ymposium on {S}ymbolic and {A}lgebraic {C}omputation. ACM,
  New York, pp. 59--66.

\bibitem[{Berthomieu et~al.(2011)Berthomieu, van~der Hoeven, and
  Lecerf}]{relaxed}
Berthomieu, J., van~der Hoeven, J., Lecerf, G., 2011. Relaxed algorithms for
  {$p$}-adic numbers. J. Th\'{e}or. Nombres Bordeaux 23~(3), 541--577.

\bibitem[{Bosma et~al.(1997)Bosma, Cannon, and Playoust}]{magma}
Bosma, W., Cannon, J., Playoust, C., 1997. The {M}agma algebra system. {I}.
  {T}he user language. J. Symbolic Comput. 24~(3-4), 235--265, computational
  algebra and number theory (London, 1993).

\bibitem[{Caruso(2017)}]{caruso}
Caruso, X., 2017. Computations with \(p\)-adic numbers. \arxivnote{1701.06794}.

\bibitem[{Cassels(1986)}]{Cassels}
Cassels, J. W.~S., 1986. Local Fields. Cambridge University Press.

\bibitem[{Dokchitser and Doris(2019)}]{conductor}
Dokchitser, T., Doris, C., 2019. 3-torsion and conductor of genus 2 curves.
  Math. Comp. 88~(318), 1913--1927.

\bibitem[{Doris(2018{\natexlab{a}})}]{exactpadicscode}
Doris, C., 2018{\natexlab{a}}. {ExactpAdics}: A package for exact \(p\)-adic
  computation. \url{https://cjdoris.github.io/ExactpAdics}.

\bibitem[{Doris(2018{\natexlab{b}})}]{exactpadics}
Doris, C., 2018{\natexlab{b}}. {ExactpAdics}: An exact representation of
  \(p\)-adic numbers. In submission. \arxivnote{1805.09794}.

\bibitem[{Doris(2018{\natexlab{c}})}]{exactpadics2code}
Doris, C., 2018{\natexlab{c}}. {ExactpAdics2}: Another package for exact
  \(p\)-adic computation. \url{https://cjdoris.github.io/ExactpAdics2}.

\bibitem[{Doris(2018{\natexlab{d}})}]{conductorcode}
Doris, C., 2018{\natexlab{d}}. {Genus2Conductor}: A package for computing the
  conductor of curves of genus 2.
  \url{https://cjdoris.github.io/Genus2Conductor}.

\bibitem[{Doris(2018{\natexlab{e}})}]{galoiscode}
Doris, C., 2018{\natexlab{e}}. {pAdicGaloisGroup}: A package for computing
  {G}alois groups of \(p\)-adic polynomials.
  \url{https://cjdoris.github.io/pAdicGaloisGroup}.

\bibitem[{Doris(2019{\natexlab{a}})}]{DPhD}
Doris, C., 2019{\natexlab{a}}. Aspects of \(p\)-adic computation. Ph.D. thesis,
  University of Bristol.

\bibitem[{Doris(2019{\natexlab{b}})}]{galoisgroups}
Doris, C., 2019{\natexlab{b}}. Computing the {G}alois group of a polynomial
  over a \(p\)-adic field, in submission.

\bibitem[{Hart et~al.(2018)Hart, Johansson, and Pancratz}]{flint}
Hart, W., Johansson, F., Pancratz, S., 2018. {FLINT}: {F}ast {L}ibrary for
  {N}umber {T}heory. \url{http://flintlib.org}.

\bibitem[{Sinclair(2015)}]{SinclairTh}
Sinclair, B., 2015. Algorithms for enumerating invariants and extensions of
  local fields. Ph.D. thesis, University of North Carolina at Greensboro.

\bibitem[{{The PARI~Group}(2020)}]{parigp}
{The PARI~Group}, 2020. {PARI/GP}. \url{http://pari.math.u-bordeaux.fr/}.

\bibitem[{{The Sage Developers}(2018)}]{sage}
{The Sage Developers}, 2018. {S}age{M}ath, the {S}age {M}athematics {S}oftware
  {S}ystem. \url{http://www.sagemath.org}.

\bibitem[{van~der Hoeven et~al.(2018)van~der Hoeven, Lecerf, and
  Mourrain}]{mathemagix}
van~der Hoeven, J., Lecerf, G., Mourrain, B., 2018. The {M}athemagix computer
  algebra and analysis system. \url{http://www.mathemagix.org}.

\end{thebibliography}
\mybibliography

\end{document}